\documentclass[12pt]{article}

\usepackage[T1]{fontenc}
\usepackage{lmodern}
\usepackage{color}
\usepackage[pdftex]{graphicx}
\usepackage{amsmath,amssymb,amsthm}
\usepackage{wrapfig}
\usepackage{ifpdf}

\newtheorem{thm}{Theorem}[section]
\newtheorem{cor}[thm]{Corollary}
\newtheorem{lemma}[thm]{Lemma}
\newtheorem{prop}[thm]{Proposition}

\newtheorem{remark}{Remark}[section]

\newtheorem{theorem}[thm]{Theorem}

\theoremstyle{definition}

\theoremstyle{plain}

\title {Borsuk--Ulam-type theorem for Stiefel manifolds and orthogonal mass partitions}

\author {Oleg R. Musin}

\begin{document}

\date{}
\maketitle

\begin{abstract}
We prove a Borsuk--Ulam-type zero theorem for the Stiefel manifold
$V_{n,k}$ with the free action of the hyperoctahedral group
$B_k=(\mathbb{Z}/2)^k\rtimes S_k$.
The theorem gives a nonvanishing criterion for explicit polynomials in
$\mathcal{P}_{n,k}=\mathbb{F}_2[a_1,\ldots,a_k]/(a_1^n,a_2^{n-1},\ldots,a_k^{n-k+1})$.

We apply this criterion to equipartitions by mutually orthogonal hyperplanes.
If $(P_{k,n})^m\ne0$ in
$\mathcal{R}_{d,k}=\mathbb{F}_2[a_1,\ldots,a_k]/(a_1^{d+1},a_2^d,\ldots,a_k^{d-k+2})$,
then any $m$ finite Borel measures in $\mathbb{R}^d$ that vanish on
affine hyperplanes admit $k$ mutually orthogonal hyperplanes such that
every $n$-element subfamily divides each measure into $2^n$ equal parts.
For the least possible dimension $\Delta^*(m,k,n)$, we obtain general
lower bounds and prove upper bounds in several cases. These give exact
values, including
$\Delta^*(2^j-1,k,2)=2^{j-1}(k+1)-1$ for all $j\ge1$ and $k\ge2$.
In the case $n=k$, the upper bound of Mani-Levitska, Vre\'cica, and
\v Zivaljevi\'c~\cite{MSZ}, originally obtained without an orthogonality
requirement, remains valid under the stronger requirement that all
$k$ hyperplanes be mutually orthogonal.
\end{abstract}

\medskip


\section{Introduction}

\subsection{Overview}

We prove a Borsuk--Ulam-type theorem for Stiefel manifolds and apply it
to orthogonal mass partitions. Theorem~\ref{thm:main-zero} gives a
polynomial criterion for every continuous $B_k$-equivariant map
\[
V_{n,k}\longrightarrow (W_1^+)^{\oplus s}
\oplus\bigoplus_{j=2}^r(W_j^+)^{\oplus m}
\]
to have a zero. The proof uses a Gysin identity
(Proposition~\ref{prop:universal-gysin}) and a relation between Schur
polynomials and the cohomology of a real Grassmannian
(Lemma~\ref{lem:schur-vandermonde}). Related Borsuk--Ulam theorems for
manifolds, products of spheres, and Stiefel manifolds are studied
in~\cite{BFHZ16,BFHZ18,BK12,CCFH20,FH88,LS25,Mus12,MusS,MusVo,MV16,ST23}.

For mass partitions, Theorem~\ref{Pknm} turns this criterion into upper
bounds on the dimension needed for an orthogonal equipartition.
We also prove a lower bound by varying the measures and counting
independent partition conditions. Together these bounds give exact
values for several infinite families of parameters. In particular,
the MVZ upper bound for full equipartitions~\cite{MSZ} still holds when
the hyperplanes are required to be mutually orthogonal. We also
consider partitions whose hyperplanes pass through every center of mass.

Section~2 develops the cohomological proof of Theorem~\ref{thm:main-zero}.
Section~3 introduces the $B_k$-equivariant functions $g_{\mu,I}$ encoding
the partition conditions.
Section~4 proves the bounds on $\Delta^*(m,k,n)$.

\subsection{Borsuk--Ulam theorem for Stiefel manifolds}

The {\em hyperoctahedral group} is
\[
 B_k=A_k\rtimes S_k, \quad A_k=(\mathbb{Z}/2)^k,
\]
where $S_k$ acts by permuting the $k$ sign generators $\lambda_1,\ldots,\lambda_k$.
The one-dimensional real representations of $A_k$ are the characters
$\lambda_I=\sum_{i\in I}\lambda_i$ for $I\subseteq[k]$.
Denote by $\mathbb{R}_{\lambda_I}$ the corresponding sign representation,
and set
\[
 W_j^+=\bigoplus_{|I|=j}\mathbb{R}_{\lambda_I}.
\]
The action extends to $B_k$ by $\sigma e_I=e_{\sigma(I)}$ for $\sigma\in S_k$.
Define the polynomials $Q_j=\prod_{|I|=j}(\sum_{i\in I}a_i)\in\mathbb{F}_2[a_1,\ldots,a_k]$.

The Stiefel manifold $V_{n,k}=\{(u_1,\ldots,u_k)\in(\mathbb{S}^{n-1})^k:\langle u_i,u_j\rangle=\delta_{ij}\}$
carries a free $B_k$-action: $A_k$ changes signs of vectors and $S_k$ permutes them.

The algebraic criterion uses the truncated polynomial ring
\[
 \mathcal{P}_{n,k}=\mathbb{F}_2[a_1,\ldots,a_k]/(a_1^n,a_2^{n-1},\ldots,a_k^{n-k+1}).
\]

\begin{theorem}\label{thm:main-zero}
Let \(n\ge k\), \(s\ge0\), \(2\le r\le k\), and \(m\ge1\).
Assume that
\begin{equation*}
 P=Q_1^s(Q_2Q_3\cdots Q_r)^m\ne0
 \quad\text{in}\quad \mathcal{P}_{n,k}.
\end{equation*}
Then every continuous \(B_k\)-equivariant map
\[
 f:V_{n,k}\longrightarrow (W_1^+)^{\oplus s}
 \oplus\bigoplus_{j=2}^r(W_j^+)^{\oplus m}
\]
has a zero.
\end{theorem}

\medskip

In characteristic two, \(Q_2\) is the Vandermonde determinant:
\[
 Q_2=\det(a_i^{j-1})_{i,j=1}^k
 =\sum_{\sigma\in S_k}\prod_{i=1}^k a_i^{\sigma(i)-1}.
\]
Modulo the defining ideal of $\mathcal P_{k,k}$, the only surviving
term is $a_1^{k-1}a_2^{k-2}\cdots a_{k-1}^1$ . Hence
$Q_2\ne0$ in $\mathcal P_{k,k}$, and Theorem~\ref{thm:main-zero} gives:

\begin{cor} \label{corVkk} Every continuous \(B_k\)-equivariant map
\[
 f:V_{k,k}\longrightarrow W_2^+
 =\bigoplus_{1\le i<j\le k}\mathbb{R}_{\lambda_i+\lambda_j}
\]
has a zero.
\end{cor}

We gave a direct proof of this corollary in~\cite{4eqpartn}.

\subsection{Hyperplane mass partitions.}

The well-known {\em ham--sandwich theorem} states:

\medskip 

\noindent {\em For every three--dimensional sandwich made from three
ingredients, there is a planar cut that simultaneously divides each
ingredient in half; thus the sandwich can be fairly divided between
two guests using a single straight cut.}

\medskip 

This theorem was proposed by Steinhaus and proved by Banach, for details see \cite{BZ}.  Stone \& Tukey \cite{ST} proved the $d$--dimensional version of the theorem in a more general setting involving measures.

Mass partition theorems are usually stated in one of two settings: discrete or continuous. The continuous versions deal with measures in $\mathbb R^d$. In this paper, unless explicitly stated otherwise, all measures are finite Borel measures such that every hyperplane has measure zero; in particular, one may take finite measures that are absolutely continuous with respect to Lebesgue measure. See \cite[Sec. 3.1]{Mat} and the surveys~\cite{RS22,Z17}.
(We also assume that  $0<\mu(\mathbb R^d)<\infty$.) 

We say that a hyperplane $H$ {\em bisects} $\mu$ (or {\em divides $\mu$ in half}\,) if 
$$
\mu(H^+)=\frac{1}{2}\,\mu(\mathbb R^d),
$$
where $H^+$ denotes one of the half--spaces defined by $H$. 

\medskip

\noindent {\bf  Ham Sandwich Theorem. }  {\em Let $\mu_1,...,\mu_d$ be measures on  $\mathbb R^d$. Then there exists a  hyperplane that simultaneously bisects all $d$ measures.}

\medskip

A discrete version of this theorem is the following; see \cite[Theorem 3.1.2]{Mat}. A hyperplane bisects a finite set if each of its two open half--spaces contains at most half of the points of the set.

\medskip 

\noindent {\bf Discrete Ham Sandwich Theorem.}  
{\em Let   $X_1,...,X_d$ be finite sets in $\mathbb R^d$. Then there exists a hyperplane that simultaneously bisects $X_1,...,X_d$.}

\medskip

Let $\Delta(m,k)$ denote the minimal dimension $d$ of Euclidean space such that for any set of $m$ measures as above in $\mathbb R^d$ there exist $k$ hyperplanes that divide each measure into $2^k$ parts of equal size. There are lower and upper bounds for this number:
\begin{equation}\tag{1.1}
\left\lceil{\left(\frac{2^k-1}{k}\right)m}\right\rceil \le \Delta(m,k) \le m+(2^{k-1}-1)2^{\lfloor \log_2 m \rfloor}.
\end{equation}
The lower bound was proved in 1996 by Ramos \cite{Ramos}, and this upper bound was obtained in 2006 by  Mani-Levitska, Vre\' cica, and \v Zivaljevi\' c \cite{MSZ}. The only instance in which lower and upper bounds coincide is in the case when
$k=1$ or $k = 2$ and $m = 2^j-1$, $j=1,2,...$

\medskip 

\subsection{Orthogonal mass partitions}

Let 
$$
P_{k,n}:=Q_1 Q_2\cdots Q_n\in \mathbb F_2[a_1,...,a_k], \quad 1\le n\le k. 
$$

The following theorem is a corollary of Theorem \ref{thm:main-zero}. 

\begin{thm} \label{Pknm} Let $1\le n\le k\le d$ and $m\ge 1$.  Suppose
$$
(P_{k,n})^m \ne 0
\quad\text{in}\quad
\mathcal{R}_{d,k}
:= \mathbb{F}_2[a_1,\ldots,a_k]\big/(a_1^{d+1}, a_2^{d},
\ldots, a_k^{d-k+2}).
$$
Then for any $m$ measures in $\mathbb R^d$ there exist $k$ mutually
orthogonal hyperplanes with every $n$-element subset dividing each
measure into $2^n$ equal parts.
\end{thm}

Thus $\mathcal R_{d,k}=\mathcal P_{d+1,k}$. The shift by one comes from
choosing the offsets to bisect the first measure before applying the
zero theorem on $V_{d,k}$.

Let $\Delta^*(m,k)$ denote the minimal dimension $d$ such that for any
$m$ measures as above in $\mathbb R^d$ there exist $k$
\emph{mutually orthogonal} hyperplanes dividing each measure into
$2^k$ equal parts.  Since orthogonality is an additional constraint,
$\Delta^*(m,k)\ge\Delta(m,k)$.  Our next result shows that the MVZ upper bound is achieved by orthogonal hyperplanes as well. It follows from Theorem \ref{Pknm}. 

\begin{thm}\label{th13}
$$
\left\lceil\frac{(2^k-1)\,m}{k}+\frac{k-1}{2}\right\rceil
\;\le\; \Delta^*(m,k) \;\le\; m+(2^{k-1}-1)\,2^{\lfloor\log_2 m\rfloor}.
$$
In particular, the bounds coincide in the following cases:
\begin{align*}
&(i)\quad k=2,\; m=2^j-1:\quad
 \Delta^*(2^j-1,2)=\Delta(2^j-1,2)=3\cdot2^{j-1}-1,\quad j\ge1;\\
&(ii)\quad k=2,\; m=2^j-2:\quad
 \Delta^*(2^j-2,2)=3\cdot2^{j-1}-2,\quad j\ge2.
\end{align*}
\end{thm}

The upper bound is the same as in~(1.1), although the class of allowed
partitions is smaller. The two exact values for $k=2$ were previously
obtained by Simon~\cite[Theorem~1.1]{Simon19}.

\medskip

More generally, let $\Delta^*(m,k,n)$ denote the minimal dimension $d$
such that for any $m$ measures in $\mathbb R^d$ there exist $k$ mutually
orthogonal hyperplanes with every $n$-element subset dividing each
measure into $2^n$ equal parts.  Clearly $\Delta^*(m,k,k)=\Delta^*(m,k)$.
The non-orthogonal variant was introduced in~\cite{BK12}; bounds for
$n=2,3$ appear in~\cite{Mejia25}.

Let $d^*(m,k,n)$ be the least integer $d\ge k$ for which
$(P_{k,n})^m\ne0$ in $\mathcal R_{d,k}$. Theorem~\ref{Pknm} gives
$\Delta^*(m,k,n)\le d^*(m,k,n)$; equality need not hold.

The case $n=1$ satisfies $\Delta^*(m,k,1)=m+k-1$ (proved in Section~4).

The algebraic forward doubling inequality $d^*(2m,k,n)\le 2\,d^*(m,k,n)$
(proved in Section~4) gives, by induction on $q$, $d^*(2^q m,k,n)\le 2^q d^*(m,k,n)$
for all $q\ge0$.  In particular, since $d^*(3,4,3)=12$ (verified by direct computation),
$$
\Delta^*(3\cdot2^q,\,4,\,3)\;\le\;d^*(3\cdot2^q,4,3)\;\le\;2^q\cdot12=12\cdot2^q,
\qquad q=0,1,2,\ldots,
$$
which is smaller than the expression $15\cdot2^q$ in~(1.4).

\medskip

The general bounds are as follows.

\begin{thm}\label{th15}
Let $\alpha_n(k)=\sum_{i=1}^n\binom{k}{i}$ and
$\beta_n(k)=\sum_{i=0}^{n-1}\binom{k-1}{i}$, with $\binom ab=0$ for $b>a$.
For all $2\le n\le k$ and $m\ge1$,
\begin{equation}\tag{1.3}
\left\lceil\frac{m\,\alpha_n(k)}{k}+\frac{k-1}{2}\right\rceil
\;\le\;
\Delta^*(m,k,n).
\end{equation}
The upper bound
\begin{equation}\tag{1.4}
\Delta^*(m,k,n)\;\le\; m+(\beta_n(k)-1)\,2^{\lfloor\log_2 m\rfloor}
\end{equation}
is established in the following cases (see Section~4):
\begin{enumerate}
\item[\rm(i)] $n=k$, any $m\ge1$ (via Dickson polynomials);
\item[\rm(ii)] $n=2$, any $k\ge2$, any $m\ge1$ (via Dyson's identity and~\cite{Mejia25});
\item[\rm(iii)] $m=2^q$ a power of two, any $2\le n\le k$ (via the Frobenius endomorphism).
\end{enumerate}
For $2<n<k$ and $m$ not a power of two, bound~(1.4) is
computationally verified for $k\le5$, $m\le11$, but remains open.
The upper bound is not always optimal: for instance,
$\Delta^*(3,4,3)=12$ while~(1.4) gives $15$.
\end{thm}

This implies Theorem~\ref{th13} (case $n=k$, $\beta_k(k)=2^{k-1}$)
and, for $n=2$, Theorem~\ref{th15b}.

\medskip

For $n=2$ the bound~(1.4) is sharp when $m=2^j-1$.
For intermediate values $2<n<k$ it is not always optimal:
the algebraic method of Theorem~\ref{Pknm} may give sharper bounds
for specific $(m,k,n)$.  For example, $\Delta^*(3,4,3)=12$
while~(1.4) gives $15$, and $\Delta^*(7,4,3)=26$ while~(1.4) gives $31$.

\medskip

For $n=2$, since $\alpha_2(k)=k(k+1)/2$ and $\beta_2(k)=k$,
Theorem~\ref{th15} gives
\begin{equation}\tag{1.2}
\left\lceil\frac{m(k+1)+k-1}{2}\right\rceil
\;\le\; \Delta^*(m,k,2) \;\le\; m+(k-1)\,2^{\lfloor\log_2 m\rfloor}.
\end{equation}
When $m=2^j-1$ both sides coincide:

\begin{thm}\label{th15b}
$$
\Delta^*(2^j-1,k,2) = 2^{j-1}(k+1)-1, \qquad j\ge1,\; k\ge2.
$$
\end{thm}

The case $j=1$ ($m=1$) gives $\Delta^*(1,k,2)=k$:

\begin{cor}\label{cor16}
For any finite measure $\mu$ in $\mathbb{R}^d$ ($d\ge2$), there exist
$d$ mutually orthogonal hyperplanes such that every pair divides $\mu$
into four equal parts.
\end{cor}

\begin{remark}
Makeev~\cite{Mak07} stated this result and outlined a proof strategy,
but the key steps were left incomplete. 
Here the result follows from Corollary~\ref{corVkk}; see also~\cite{4eqpartn}.
\end{remark}

\medskip

Let $\Delta^\oplus(m,k,n)$ denote the centroid variant of $\Delta^*(m,k,n)$:
the minimal $d$ such that for any $m$ measures in~$\mathbb{R}^d$ with
finite first moments there exist $k$ mutually orthogonal hyperplanes,
each passing through all $m$ centers of mass, with every $n$-element
subfamily partitioning each measure into $2^n$ equal parts.

Requiring all hyperplanes to pass through a prescribed point adds one
copy of $W_1^+$ to the test representation. Applying this to the $m$
centers of mass gives the following variant of Theorem~\ref{Pknm}.

\begin{thm}\label{th17}
Let $1\le n\le k\le d$ and $m\ge1$.  Suppose
$$
(Q_1\,P_{k,n})^m \ne 0
\quad\text{in}\quad
\mathcal{R}_{d,k}.
$$
Then for any $m$ measures in $\mathbb{R}^d$ with finite first moments
there exist $k$ mutually orthogonal hyperplanes, each passing through
the centers of mass of all $m$ measures, such that every $n$-element
subfamily partitions each measure into $2^n$ equal parts.

Multiplication by $Q_1^m$ adds $m$ to every exponent. Hence
$(P_{k,n})^m\ne0$ in $\mathcal R_{D,k}$ implies
$(Q_1P_{k,n})^m\ne0$ in $\mathcal R_{D+m,k}$, and therefore
\[
\Delta^\oplus(m,k,n)\le d^*(m,k,n)+m.
\]
Setting $L=\lceil m\alpha_n(k)/k+(k-1)/2\rceil$, the lower bound
$\Delta^\oplus(m,k,n)\ge L+m$ is proved in Section~4.
For $n=2$, $m=2^j-1$ the bounds are sharp:
$$
\Delta^\oplus(2^j-1,k,2)=2^{j-1}(k+3)-2.
$$
\end{thm}

\subsection{Algorithms for mass partitioning.}

In discrete versions of mass-partitioning results, the task involves partitioning several finite families of points in $\mathbb{R}^d$ in a prescribed manner. If the total number of points is $N$, it is desirable to have an algorithm for finding such a partition whose running time is expressed in terms of $N$.

In two dimensions, let $X_1$ and $X_2$ be finite point sets with $N$ points in total. Edelsbrunner and Waupotitsch \cite{EW} gave an $O(N\log N)$--time algorithm for computing a ham--sandwich cut. Lo, Matoušek, and Steiger \cite{LoMS} proved that a planar ham--sandwich cut can be computed in $O(N)$ time. They also presented polynomial algorithms for finding ham--sandwich cuts in every fixed dimension $d>1$.

\medskip

A previous $O(N\log N)$--time algorithm for the pancake theorem was given in \cite{RoySt}.
Recently, we improved this result:

\medskip

\noindent {\bf Theorem} \cite{FM25}. {\em For any set of $N$ points $P$ in the plane, two orthogonal lines for which each open quadrant contains at most $N/4$ points can be found in $\Theta(N)$ time.}

\medskip

In  \cite{FM25} we also proved that the computational complexity of the discrete versions of Theorem \ref{th15} (the case where $n = k = 2$) and Corollary \ref{cor16} is polynomial. However, the algorithms proposed in \cite{FM25} are not optimal. Finding optimal algorithms for general $m,k,n$ remains open.

\medskip 

\medskip

\noindent{\bf Acknowledgment.} I would like to thank Pavle
Blagojevi\'c, Mikhail Bludov, Florian Frick, Roman Karasev,
Andrey Ryabichev, Pablo Sober\'on, and Rade \v Zivaljevi\'c
for discussions, comments, and references.

\section{The hyperoctahedral zero theorem}

All cohomology groups in this section have coefficients in $\mathbb F_2$.
Euler classes mean top Stiefel--Whitney classes.

We first compute a Gysin image on the classifying space of $B_k$.
The Schur--Vandermonde lemma then relates its restriction to the
Grassmannian to the polynomial condition in Theorem~\ref{thm:main-zero}.

Let
\[
 p:BB_k\longrightarrow BO(k)
\]
be induced by the standard inclusion \(B_k\subset O(k)\). Its fiber is
\[
 F_k=O(k)/B_k,
 \qquad \dim F_k=\binom{k}{2}.
\]
For \(1\le j\le k\), let
\[
 e_j=w_{\binom{k}{j}}
 \bigl(EB_k\times_{B_k}W_j^+\bigr)
 \in H^{\binom{k}{j}}(BB_k;\mathbb F_2).
\]
Its restriction to \(BA_k\) is \(Q_j\). If
\(R=\prod_{j=1}^kQ_j^{c_j}\), with $c_j\ge0$, put
\[
 W_R=\bigoplus_{j=1}^k(W_j^+)^{\oplus c_j},
 \qquad
 e_R=w_{\dim W_R}(EB_k\times_{B_k}W_R).
\]
The same symmetric polynomial defines a characteristic class
\(R(\gamma_k)\in H^*(BO(k);\mathbb F_2)\), where \(\gamma_k\) is the
universal real \(k\)-plane bundle.

\begin{lemma}\label{lem:fixed-points}
The fixed-point set \(F_k^{A_k}\) is finite. It contains a distinguished
point \(x_0\), represented by the coordinate-line decomposition
\[
 \mathbb R^k=\mathbb R e_1\oplus\cdots\oplus\mathbb R e_k.
\]
At this point
\[
 e_{A_k}(T_{x_0}F_k)=Q_2,
 \qquad e_2|_{x_0}=Q_2.
\]
At every other point \(x\in F_k^{A_k}\), one has \(e_2|_x=0\).
\end{lemma}

\begin{proof}
A point of \(F_k\) is an unordered orthogonal decomposition of
\(\mathbb R^k\) into lines. It is fixed by \(A_k\) precisely when
\(A_k\) permutes those lines. Such a point determines, up to
conjugation in \(B_k\), a homomorphism \(A_k\to B_k\) whose signed
permutation representation is equivalent to the standard diagonal
representation. There are only finitely many such homomorphisms. For
each of them, the space of orthogonal intertwiners is a coset of the
centralizer of \(A_k\) in \(O(k)\). Since the standard \(A_k\)-module is
the sum of \(k\) distinct one-dimensional characters, this centralizer
is the finite diagonal sign group. Hence \(F_k^{A_k}\) is finite.

At the coordinate point,
\[
 T_{x_0}F_k\cong\mathfrak o(k)
 \cong\bigoplus_{i<j}\mathbb R_{\lambda_i+\lambda_j}
\]
as an \(A_k\)-module. The same weights occur in \(W_2^+|_{A_k}\), so
both equivariant Euler classes are \(Q_2\).

Now let \(x\) correspond to an \(A_k\)-invariant line decomposition
\(L=(L_1,\ldots,L_k)\). There is an \(A_k\)-invariant splitting
\[
 \operatorname{Sym}^2(\mathbb R^k)=D_L\oplus O_L,
 \quad
 D_L=\bigoplus_i\operatorname{Sym}^2(L_i),
 \quad
 O_L=\bigoplus_{i<j}L_i\mathbin\odot L_j.
\]
The second summand is the fiber representation of \(W_2^+\) at \(x\).
Because the standard \(A_k\)-representation is multiplicity-free,
\[
 \dim\operatorname{Sym}^2(\mathbb R^k)^{A_k}=k.
\]
If \(s\) is the number of \(A_k\)-orbits on the set of lines, then
\(D_L\) is the corresponding permutation module, and hence
\[
 \dim D_L^{A_k}=s,
 \qquad \dim O_L^{A_k}=k-s.
\]
If \(s=k\), every line is individually \(A_k\)-invariant, and the
multiplicity-free decomposition forces the lines to be the coordinate
character lines. Thus \(x=x_0\). Every nonstandard fixed point has
\(s<k\), so \(O_L\) contains a trivial \(A_k\)-summand and its
equivariant Euler class is zero.
\end{proof}

\begin{prop}[Universal Gysin identity]\label{prop:universal-gysin}
For every polynomial \(R=\prod_{j=1}^kQ_j^{c_j}\) with $c_j\ge0$,
\[
 p_!(e_2e_R)=R(\gamma_k).
\]
\end{prop}

\begin{proof}
Restrict the bundle \(p\) along the diagonal inclusion
\(i:BA_k\to BO(k)\), and invert all nonzero linear forms in
\[
 H^*(BA_k;\mathbb F_2)=\mathbb F_2[a_1,\ldots,a_k].
\]
The localization theorem and the fixed-point integration formula give
\[
 i^*p_!(e_2e_R)
 =\sum_{x\in F_k^{A_k}}
 \frac{(e_2e_R)|_x}{e_{A_k}(T_xF_k)}.
\]
See \cite{Borel1960,Hsiang1975,tomDieck1979}. By Lemma
\ref{lem:fixed-points}, the contribution of \(x_0\) is
\[
 \frac{Q_2R}{Q_2}=R,
\]
and every other contribution vanishes. Hence
\(i^*p_!(e_2e_R)=R\). Localization is injective on the polynomial ring.
Finally, the injective map
\[
 H^*(BO(k);\mathbb F_2)=\mathbb F_2[w_1,\ldots,w_k]
 \longrightarrow \mathbb F_2[a_1,\ldots,a_k]
\]
sends \(w_j\) to the \(j\)-th elementary symmetric polynomial and has
the symmetric polynomials as its image. Therefore the last equality is
precisely \(p_!(e_2e_R)=R(\gamma_k)\).
\end{proof}

\begin{lemma}[Schur--Vandermonde lemma]
\label{lem:schur-vandermonde}
For every symmetric polynomial
\(R\in\mathbb F_2[a_1,\ldots,a_k]^{S_k}\),
\[
 Q_2R\ne0\text{ in }\mathcal{P}_{n,k}
 \quad\Longleftrightarrow\quad
 R(\gamma_k)\ne0
 \text{ in }H^*(\operatorname{Gr}_k(\mathbb R^n);\mathbb F_2).
\]
\end{lemma}

\begin{proof}
Expand \(R\) in the Schur basis:
\[
 R=\sum_\lambda c_\lambda s_\lambda,
 \qquad c_\lambda\in\mathbb F_2.
\]
Over \(\mathbb F_2\), the Vandermonde alternant and the alternant formula
for Schur polynomials give
\[
 Q_2=\det(a_i^{k-j})_{i,j=1}^k,
 \qquad
 Q_2s_\lambda
 =\det(a_i^{\lambda_j+k-j})_{i,j=1}^k.
\]
The exponents
\[
 \lambda_1+k-1>\lambda_2+k-2>\cdots>\lambda_k
\]
are strictly decreasing. Therefore different partitions have disjoint
monomial supports, so no cancellation between them is possible. The
last determinant survives in \(\mathcal{P}_{n,k}\) exactly when its
decreasing exponent sequence can be matched with
\[
 n-1,n-2,\ldots,n-k.
\]
This is equivalent to \(\lambda_1\le n-k\). Consequently,
\(Q_2R\ne0\) in \(\mathcal{P}_{n,k}\) exactly when some coefficient
\(c_\lambda\) with \(\lambda\subseteq(n-k)^k\) is nonzero.

The mod-two cohomology of the real Grassmannian has the Schubert basis
\[
 \{s_\lambda(\gamma_k):\lambda\subseteq(n-k)^k\}.
\]
The Grassmannian presentation also gives
$s_\lambda(\gamma_k)=0$ when $\lambda_1>n-k$.
Thus the same condition is equivalent to \(R(\gamma_k)\ne0\). See
\cite{MilnorStasheff1974,Fulton1997} for the Grassmannian presentation
and the Schur alternant identities.
\end{proof}

\begin{lemma}[Euler classes of character representations]\label{lem:euler-characters}
Let $q=\varepsilon_1\lambda_1+\cdots+\varepsilon_k\lambda_k$ with $\varepsilon_i\in\mathbb{F}_2$,
and let $L_q=EA_k\times_{A_k}\mathbb{R}_q$ be the real line bundle over $BA_k$
associated with the character~$q$.  Then
$$
e_{A_k}(\mathbb{R}_q)=w_1(L_q)=\varepsilon_1a_1+\cdots+\varepsilon_ka_k
\quad\in H^1(BA_k;\mathbb{F}_2)=\mathbb{F}_2[a_1,\ldots,a_k]^1.
$$
Consequently,
$$
e_{A_k}(W_j^+)=\prod_{\substack{I\subseteq[k]\\|I|=j}}\!\!\left(\sum_{i\in I}a_i\right)=Q_j,
\qquad
e_{A_k}\!\left((W_1^+)^{\oplus s}\oplus\bigoplus_{j=2}^r(W_j^+)^{\oplus m}\right)
=Q_1^s(Q_2Q_3\cdots Q_r)^m.
$$
\end{lemma}

\begin{proof}
The character line bundle satisfies
$L_q\cong L_{\lambda_1}^{\otimes\varepsilon_1}\otimes\cdots\otimes L_{\lambda_k}^{\otimes\varepsilon_k}$.
For real line bundles,
$w_1(L\otimes L')=w_1(L)+w_1(L')$
in cohomology with $\mathbb F_2$ coefficients.
Therefore $w_1(L_q)=\sum_i\varepsilon_i\,w_1(L_{\lambda_i})=\sum_i\varepsilon_i a_i$.
The remaining assertions follow from multiplicativity of the top Stiefel--Whitney
class under direct sums.
\end{proof}

\begin{proof}[Proof of Theorem~\ref{thm:main-zero}]
Write $P=Q_2R$ where $R=Q_1^sQ_2^{m-1}Q_3^m\cdots Q_r^m$ is symmetric (as a product of symmetric polynomials $Q_j$).
Let $X_{n,k}=V_{n,k}/B_k$, and let
$c:X_{n,k}\to BB_k$ classify the principal $B_k$-bundle $V_{n,k}\to X_{n,k}$.
Sending an orthonormal frame to its span gives a smooth fiber bundle
$\pi:X_{n,k}\to\operatorname{Gr}_k(\mathbb{R}^n)$ with fiber $O(k)/B_k$.
The classifying map $c_\gamma:\operatorname{Gr}_k(\mathbb{R}^n)\to BO(k)$
of the tautological bundle fits into a homotopy Cartesian square
\[
\begin{array}{ccc}
X_{n,k} & \xrightarrow{\ c\ } & BB_k \\
\pi\downarrow & & \downarrow p \\
\operatorname{Gr}_k(\mathbb{R}^n) & \xrightarrow{c_\gamma} & BO(k),
\end{array}
\]
where $p:BB_k\to BO(k)$ is the map induced by the standard representation.

Let $\mathcal{W}=V_{n,k}\times_{B_k}\bigl((W_1^+)^{\oplus s}\oplus\bigoplus_{j=2}^r(W_j^+)^{\oplus m}\bigr)$
be the associated bundle over $X_{n,k}$, and set $\bar e_j=c^*(e_j)$
for the pullback of the universal classes $e_j\in H^*(BB_k;\mathbb{F}_2)$.
By Lemma~\ref{lem:euler-characters} and naturality,
\[
 w_{\mathrm{top}}(\mathcal W)=\bar e_1^s\bar e_2^m\cdots\bar e_r^m=\bar e_2\bar e_R.
\]
Base-change naturality of the Gysin map and Proposition~\ref{prop:universal-gysin} give
\[
 \pi_!(w_{\mathrm{top}}(\mathcal W))
 =\pi_!(\bar e_2\bar e_R)
 =c_\gamma^*p_!(e_2 e_R)
 =c_\gamma^* R(\gamma_k)=R(\gamma_k).
\]
The hypothesis $P\ne0$ in $\mathcal{P}_{n,k}$ and Lemma~\ref{lem:schur-vandermonde}
give $R(\gamma_k)\ne0$, hence $w_{\mathrm{top}}(\mathcal{W})\ne0$.

A $B_k$-equivariant map $f$ in the statement descends to a section of $\mathcal{W}$.
If $f$ had no zero, the section would be nowhere zero, forcing
$w_{\mathrm{top}}(\mathcal{W})=0$, a contradiction.
\end{proof}

\section{Equipartition of measures: the functions \(g_\mu\)}

\subsection{Oriented hyperplanes in \(\mathbb R^d\)}

Put \(e_0=(1,0,\ldots,0)\in\mathbb R^{d+1}\).  For
\[
 v=(t_0,t_1,\ldots,t_d)=(t_0,w)\in\mathbb S^d,
 \qquad w=(t_1,\ldots,t_d),
\]
define
\[
 H(v)=\{x\in\mathbb R^d:\langle x,w\rangle\le t_0\}.
\]
Equivalently, given a unit normal $u\in\mathbb{S}^{d-1}$ and offset
$t\in\mathbb{R}$, the corresponding sphere point is
\[
 v = \frac{(t,u)}{\sqrt{1+t^2}}\in\mathbb{S}^d,
\]
so $w=u/\sqrt{1+t^2}$ and $t_0=t/\sqrt{1+t^2}$; the hyperplane
$H_0(v)$ is $\{x:\langle x,u\rangle=t\}$.
If \(w\ne0\), let
\[
 H_0(v)=\{x\in\mathbb R^d:\langle x,w\rangle=t_0\}.
\]
The antipodal points \(v\) and \(-v\) represent the same affine
hyperplane with the two orientations interchanged:
\[
 H_0(-v)=H_0(v),
 \qquad
 H(-v)=\{x:\langle x,w\rangle\ge t_0\}.
\]
At the two exceptional points,
\[
 H(e_0)=\mathbb R^d,
 \qquad H(-e_0)=\varnothing,
\]
whereas \(H_0(\pm e_0)\) is not defined.

As in Section 1, all measures are finite nonzero Borel measures for
which every affine hyperplane has measure zero.  Consequently,
\(H(v)\) and \(H(-v)\) partition the measure whenever
\(v\ne\pm e_0\).  The hyperplane \(H_0(v)\) bisects \(\mu\)
precisely when
\[
 \mu(H(v))=\mu(H(-v))=\frac12\mu(\mathbb R^d).
\]
In particular, a parameter \(v\) satisfying the bisection condition
cannot equal \(\pm e_0\).

For \(k\ge1\), put
\[
 S_{d,k}=(\mathbb S^d)^k.
\]
A point \(\mathbf v=(v_1,\ldots,v_k)\in S_{d,k}\) is an ordered
collection of oriented hyperplane parameters.

\subsection{The functions \(g_\mu\)}

Let
\[
 C^n=\{-1,+1\}^n.
\]
For \(\mathbf v=(v_1,\ldots,v_n)\in S_{d,n}\) and
\(\mathbf s=(s_1,\ldots,s_n)\in C^n\), define
\[
 H(\mathbf v,\mathbf s)
 =\bigcap_{i=1}^n H(s_iv_i),
 \qquad
 m_{\mathbf v}(\mathbf s)
 =\mu\bigl(H(\mathbf v,\mathbf s)\bigr).
\]
If \(r(\mathbf s)\) denotes the number of negative coordinates of
\(\mathbf s\), put
\[
 g_\mu^{(n)}(\mathbf v)
 =\sum_{\mathbf s\in C^n}
 (-1)^{r(\mathbf s)}m_{\mathbf v}(\mathbf s).
\]

Let \(A_n=(\mathbb Z/2)^n\), with generators
\(\lambda_1,\ldots,\lambda_n\), act on \(S_{d,n}\) by changing the
signs of the corresponding coordinates, and put
\[
 \omega_n=\lambda_1+\cdots+\lambda_n.
\]

\begin{lemma}\label{L31}
The function
\[
 g_\mu^{(n)}:S_{d,n}\longrightarrow\mathbb R_{\omega_n}
\]
is continuous and \(A_n\)-equivariant.  It is also symmetric in
\(v_1,\ldots,v_n\).
\end{lemma}

\begin{proof}
If \(\mathbf v^{(r)}\to\mathbf v\), then for \(\mu\)-almost every
\(x\), the indicator functions of
\(H(\mathbf v^{(r)},\mathbf s)\) converge to that of
\(H(\mathbf v,\mathbf s)\).  The only possible exceptional points lie
on one of the limiting hyperplanes \(H_0(v_i)\); these have
\(\mu\)-measure zero.  At \(v_i=\pm e_0\), the defining inequality is
eventually strict for each fixed \(x\), so the same conclusion holds.
Continuity therefore follows from dominated convergence.

Replacing \(v_i\) by \(-v_i\) interchanges the terms indexed by
\(s_i=1\) and \(s_i=-1\), and hence changes the sign of
\(g_\mu^{(n)}\).  This is precisely the character
\(\omega_n\).  Finally, a permutation of the vectors \(v_i\) merely
permutes the summands in the definition, proving symmetry.
\end{proof}

For \(I=\{i_1<\cdots<i_j\}\subset[k]\), write
\[
 g_{\mu,I}(\mathbf v)
 =g_\mu^{(j)}(v_{i_1},\ldots,v_{i_j}).
\]
Let \(B_k\) act on \(S_{d,k}\) by independent sign changes and
permutations of the \(k\) factors.  Recall that
\[
 W_j^+
 =\bigoplus_{\substack{I\subset[k]\\|I|=j}}\mathbb R_{\lambda_I},
 \qquad
 \lambda_I=\sum_{i\in I}\lambda_i.
\]
Lemma \ref{L31} implies that
\[
 \mathcal G_{\mu,j}:S_{d,k}\longrightarrow W_j^+,
 \qquad
 \mathcal G_{\mu,j}(\mathbf v)
 =\sum_{\substack{I\subset[k]\\|I|=j}}
 g_{\mu,I}(\mathbf v)e_I,
\]
is \(B_k\)-equivariant: sign changes act through the characters
\(\lambda_I\), and permutations send the coordinate indexed by \(I\)
to the coordinate indexed by \(\sigma(I)\).

\begin{lemma}\label{L32}
Let \(\mu\) be a finite measure in \(\mathbb R^d\), and let
\(\mathbf v=(v_1,\ldots,v_n)\in S_{d,n}\).  Suppose that
\[
 g_{\mu,I}(\mathbf v)=0
 \qquad\text{for every nonempty subset }I\subset[n].
\]
Then the hyperplanes \(H_0(v_1),\ldots,H_0(v_n)\) divide \(\mu\)
into \(2^n\) equal parts.
\end{lemma}

\begin{proof}
The equations with \(|I|=1\) imply that every \(H_0(v_i)\) bisects
\(\mu\); in particular, all these hyperplanes are defined.  Since
their boundaries have measure zero, the numbers
\(m_{\mathbf v}(\mathbf s)\), \(\mathbf s\in C^n\), are the measures
of the \(2^n\) cells of the arrangement.

For \(I\subset[n]\), define the Walsh coefficient
\[
 \widehat m(I)
 =\sum_{\mathbf s\in C^n}
 \left(\prod_{i\in I}s_i\right)m_{\mathbf v}(\mathbf s).
\]
Summing first over the signs with indices outside \(I\) gives
\[
 \widehat m(I)=g_{\mu,I}(\mathbf v)
 \quad\text{for }I\ne\varnothing,
 \qquad
 \widehat m(\varnothing)=\mu(\mathbb R^d).
\]
Walsh inversion on \(C^n\) yields
\[
 m_{\mathbf v}(\mathbf s)
 =2^{-n}\sum_{I\subset[n]}
 \left(\prod_{i\in I}s_i\right)\widehat m(I).
\]
All nonconstant coefficients vanish by hypothesis, and therefore
\[
 m_{\mathbf v}(\mathbf s)
 =2^{-n}\mu(\mathbb R^d)
 \qquad\text{for every }\mathbf s\in C^n.
\]
\end{proof}

The \(2^n-1\) conditions in Lemma~\ref{L32} are linearly independent
as functionals of the \(2^n\) cell masses (they are the nonconstant Walsh coefficients).

\begin{lemma}\label{L33}
Let \(1\le n\le k\le d\), let \(\mu\) be a finite measure in
\(\mathbb R^d\), and let
\(\mathbf v=(v_1,\ldots,v_k)\in S_{d,k}\).  Suppose that
\begin{equation}\tag{3.1}\label{eq:partition-one}
 g_{\mu,I}(\mathbf v)=0
 \qquad
 \text{for every }I\subset[k]\text{ with }1\le|I|\le n.
\end{equation}
Then every \(n\)-element subset of
\(H_0(v_1),\ldots,H_0(v_k)\) divides \(\mu\) into \(2^n\) equal
parts.  Moreover, \eqref{eq:partition-one} consists of
\[
 \alpha_n(k)=\sum_{j=1}^n\binom{k}{j}
\]
linearly independent conditions on the \(2^k\) cell masses.
\end{lemma}

\begin{proof}
There is one equation for each subset \(I\subset[k]\) with
\(1\le|I|\le n\), giving the stated number.  These equations are
distinct nonconstant Walsh coefficients and hence are linearly
independent as functionals of the cell masses.  For any
\(J\subset[k]\) with \(|J|=n\), the equations indexed by the nonempty
subsets of \(J\) are exactly the hypotheses of Lemma \ref{L32}.
Applying that lemma to every such \(J\) proves the assertion.
\end{proof}

The corresponding statement for several measures is immediate.

\begin{thm}\label{th34}
Let \(1\le n\le k\le d\), and let
\(\mu_1,\ldots,\mu_m\) be finite measures in \(\mathbb R^d\).
Suppose that \(\mathbf v=(v_1,\ldots,v_k)\in S_{d,k}\) satisfies
\begin{equation}\tag{3.2}\label{eq:partition-many}
 g_{\mu_\ell,I}(\mathbf v)=0,
 \qquad
 \ell=1,\ldots,m,\quad
 I\subset[k],\quad 1\le|I|\le n.
\end{equation}
Then every \(n\)-element subset of the hyperplanes
\(H_0(v_1),\ldots,H_0(v_k)\) divides each measure \(\mu_\ell\) into
\(2^n\) equal parts.
\end{thm}

\begin{proof}
Apply Lemma \ref{L33} separately to each measure \(\mu_\ell\).
\end{proof}

\subsection{Additional equivariant constraints}

Let \(p\in\mathbb R^d\), and write \(v=(t_0,w)\in\mathbb S^d\).
Define
\[
 g_p(v)=t_0-\langle w,p\rangle.
\]
Then \(g_p(-v)=-g_p(v)\), and \(g_p(v)=0\) precisely when
\(H_0(v)\) passes through \(p\).  Consequently,
\[
 \mathcal P_p:S_{d,k}\longrightarrow W_1^+,
 \qquad
 \mathcal P_p(v_1,\ldots,v_k)
 =\sum_{i=1}^k g_p(v_i)e_{\{i\}},
\]
is \(B_k\)-equivariant.

For a set \(S=\{p_1,\ldots,p_\ell\}\subset\mathbb R^d\), consider the
additional equations
\begin{equation}\tag{3.3}\label{eq:point-constraints}
 g_{p_a}(v_i)=0,
 \qquad a=1,\ldots,\ell,\quad i=1,\ldots,k.
\end{equation}

\begin{cor}\label{Cor3S}
If a configuration \(\mathbf v\in S_{d,k}\) satisfies
\eqref{eq:partition-many} and \eqref{eq:point-constraints}, then its
hyperplanes divide each of the \(m\) measures as in Theorem
\ref{th34}, and every hyperplane passes through every point of \(S\).
\end{cor}

\begin{proof}
The equipartition conclusion follows from Theorem \ref{th34}, and the
incidence conclusion follows from the definition of \(g_p\).
\end{proof}

Orthogonality is encoded by another equivariant map.  Write
\[
 v_i=(t_{i0},w_i)\in\mathbb S^d,
 \qquad w_i\in\mathbb R^d,
\]
and define
\[
\begin{aligned}
 \mathcal O&:S_{d,k}\longrightarrow W_2^+,\\
 \mathcal O(v_1,\ldots,v_k)
 &=\sum_{1\le i<j\le k}
 \langle w_i,w_j\rangle e_{\{i,j\}}.
\end{aligned}
\]
Changing the sign of \(v_i\) changes the sign of every coordinate
indexed by a pair containing \(i\), and permutations act by
permuting the pairs.  Thus \(\mathcal O\) is \(B_k\)-equivariant.
When the equations with \(|I|=1\) in
\eqref{eq:partition-many} hold, all \(w_i\) are nonzero.  In that
case, \(\mathcal O(\mathbf v)=0\) precisely when the hyperplanes
\(H_0(v_1),\ldots,H_0(v_k)\) are mutually orthogonal.

\begin{cor}\label{Cor35}
If a configuration \(\mathbf v\in S_{d,k}\) satisfies
\eqref{eq:partition-many} and \(\mathcal O(\mathbf v)=0\), then its
hyperplanes are mutually orthogonal and every \(n\)-element subset
divides each of the \(m\) measures into \(2^n\) equal parts.
\end{cor}

\begin{proof}
This follows from Theorem \ref{th34} and the preceding description of
the zero set of \(\mathcal O\).
\end{proof}

\section{Proof of the theorem on orthogonal partitions}

\subsection{The case $n=1$: ham-sandwich with orthogonal normals}

\begin{prop}\label{prop:n1}
$\Delta^*(m,k,1)=m+k-1$.
\end{prop}

\begin{proof}
Put $d=m+k-1$. Suppose $i-1$ orthonormal normals
$u_1,\ldots,u_{i-1}$ have been chosen. Their orthogonal complement $E$
has dimension $d-i+1\ge m$. Push each measure forward under the
orthogonal projection onto $E$. The resulting measures vanish on
hyperplanes, since the inverse image of a hyperplane in $E$ is a
hyperplane in $\mathbb R^d$. The ham-sandwich theorem in $E$ gives a
unit normal $u_i\in E$ and an offset $t_i$ that bisect all $m$ projected
measures, and hence all the original measures. Repeating this for
$i=1,\ldots,k$ proves the upper bound.

For the lower bound, $d\ge k$ already proves the case $m=1$.
If $m\ge2$, take Gaussian probability measures with centers
$0,e_1,\ldots,e_{m-1}$ in $\mathbb R^{m+k-2}$, where the $e_i$ are
standard basis vectors. A hyperplane bisects one of these measures
exactly when it passes through its center. Thus every common bisecting
normal is orthogonal to $\operatorname{span}(e_1,\ldots,e_{m-1})$.
This orthogonal complement has dimension $k-1$, so it cannot contain
$k$ orthonormal normals.
\end{proof}

\subsection{Lower bound}

\begin{lemma} \label{L41}
For $2\le n\le k$,
$$
\Delta^*(m,k,n)\ge \left\lceil{\frac{m\,\alpha_n(k)}{k}+\frac{k-1}{2}}\right\rceil.
$$
\end{lemma}

\begin{proof}
We show that for $d$ below the stated threshold there exist $m$ admissible
measures in $\mathbb{R}^d$ admitting no configuration of $k$ mutually
orthogonal hyperplanes that $n$-wise equipartitions all of them.

\medskip

\noindent{\bf Step 1: Configuration space and equations.}
We may assume $d\ge k$, since otherwise $k$ orthogonal normals cannot exist.
Let $X = V_{d,k}\times\mathbb{R}^k$ with
$$
D = \dim X = dk - \tfrac{k(k-1)}{2},
$$
and let $\mathcal{I}=\{I\subset[k]:1\le|I|\le n\}$,
$a=|\mathcal{I}|=\alpha_n(k)$.
A point $z=(u,t)\in X$ determines $k$ affine hyperplanes
$H_i(z)=\{x:\langle u_i,x\rangle=t_i\}$ with mutually orthogonal
normals.  For a probability density $\rho_\ell$ and $I\in\mathcal{I}$,
the Walsh coordinate is
$$
F_{\ell,I}(\rho,z)
=\int_{\mathbb{R}^d}
\prod_{i\in I}\mathrm{sgn}(\langle u_i,x\rangle-t_i)\,\rho_\ell(x)\,dx.
$$
By Walsh inversion (Lemma~\ref{L33}), vanishing of all
$F_{\ell,I}(\rho,z)=0$ for $1\le|I|\le n$ is equivalent to
$n$-wise equipartition of $\mu_\ell$ by the hyperplanes.
There are $N=ma$ target coordinates in total.

\medskip

\noindent{\bf Step 2: Compact zero set.}
Let $\gamma_d=(2\pi)^{-d/2}e^{-\|x\|^2/2}$ be the standard Gaussian.
For densities with $\|\rho_\ell-\gamma_d\|_1<1/8$, every bisecting
hyperplane satisfies $|t_i|<2$.  Hence every zero of the joint test map
lies in the compact set $K=V_{d,k}\times[-2,2]^k\subset X$.

\medskip

\noindent{\bf Step 3: Explicit local perturbations.}
Fix $z=(u,t)\in K$.  For $s\in\{-1,1\}^k$,
place a smooth radially symmetric bump $\eta_s$ of mass~$1$
supported in a ball of radius $\tfrac{1}{8}$ about
$x_s=\sum_i(t_i+s_i)u_i$.  Define
$$
h^z_I(x) = 2^{-k}\sum_{s\in\{-1,1\}^k}
\prod_{i\in I}s_i\cdot\eta_s(x).
$$
Walsh orthogonality gives $\int h^z_I\,dx=0$ and
$\int\phi_J(z,x)\,h^z_I(x)\,dx=\delta_{IJ}$ for all $I,J\in\mathcal{I}$,
where $\phi_J(z,x)=\prod_{i\in J}\mathrm{sgn}(\langle u_i,x\rangle-t_i)$.
Thus the derivative of $F_{\ell,I}$ with respect to the $\ell$-th
density in the direction $h^z_J$ is $\delta_{IJ}$~: the
$a\times a$ Jacobian in the density variables is the identity.

\medskip

\noindent{\bf Step 4: Finite cover and admissible densities.}
The bump supports remain in the same open orthants on a sufficiently
small neighborhood $U_z$ of $z$, so the identity Jacobian persists there.  Cover $K$ by finitely many such neighborhoods
$U_{z_1},\ldots,U_{z_L}$, and let $\{h_1,\ldots,h_{La}\}$ be the
union of these $La$ fixed perturbation functions (one set of $a$ functions
per covering point). Put $U=\bigcup_{j=1}^L U_{z_j}$. Each measure $\mu_\ell$ gets its own independent
copy of these $La$ functions, giving $mLa$ coefficient variables in total.
Set
$$
M_{\mathrm{tot}}=\sup_x\sum_{q=1}^{La}|h_q(x)|,\quad
\Lambda_{\mathrm{tot}}=\sum_{q=1}^{La}\|h_q\|_1,
$$
and $\varepsilon_0=\min\!\bigl(\beta/(2M_{\mathrm{tot}}),\,1/(8\Lambda_{\mathrm{tot}})\bigr)$,
where $\beta=\min_S\gamma_d>0$ and $S$ is the union of supports.
For coefficient arrays $c=(c_{\ell q})\in(-\varepsilon_0,\varepsilon_0)^{mLa}$, set
$$
\rho_{\ell,c}(x)=\gamma_d(x)+\sum_{q=1}^{La} c_{\ell q}h_q(x).
$$
The choice of $M_{\mathrm{tot}}$ ensures each $\rho_{\ell,c}$ is
strictly positive; the choice of $\Lambda_{\mathrm{tot}}$ ensures
$\|\rho_{\ell,c}-\gamma_d\|_1<1/8$.  Each $\rho_{\ell,c}$ is a smooth
probability density giving every hyperplane measure zero.

\medskip

\noindent{\bf Step 5: Choosing the measures.}
The universal test map
$\mathcal{T}:(-\varepsilon_0,\varepsilon_0)^{mLa}\times U\to\mathbb{R}^N$,
$(c,z)\mapsto(F_{\ell,I}(\rho_c,z))$,
is smooth. To see this locally, extend the frame to an orthogonal basis
and translate by $\sum_i t_i u_i$. This makes the domains of integration
fixed orthants; differentiating the compactly supported perturbations
under the integral is then valid. The Gaussian term is smooth as well,
since its projections onto orthogonal normals are independent standard
Gaussians. The derivative of $\mathcal T$ in the $c$-variables is surjective
everywhere (different measures have independent coefficient blocks, and
each block has identity Jacobian by Step~3).  So $\mathcal{T}$ is a
submersion and its zero set $\mathcal{Z}$ is a smooth manifold of
dimension $mLa+D-N$.

Suppose $N>D$.  Then $\dim\mathcal{Z}<mLa=\dim(-\varepsilon_0,\varepsilon_0)^{mLa}$,
so the projection $\pi:\mathcal{Z}\to(-\varepsilon_0,\varepsilon_0)^{mLa}$
has every point critical.  By Sard's theorem, $\pi(\mathcal{Z})$ has
Lebesgue measure zero.  Choosing $c$ outside $\pi(\mathcal{Z})$ gives
$m$ admissible measures with no equipartitioning configuration.

The condition $N>D$ means $m\,\alpha_n(k)>dk-k(k-1)/2$, i.e.\
$d<m\,\alpha_n(k)/k+(k-1)/2$, proving the lower bound.
\end{proof}

Recall from Section~1 that for $1\le j\le k$ we have set
\[
 Q_j = Q_j(a_1,\ldots,a_k)
 :=\prod_{\substack{I\subset[k]\\|I|=j}}\!\!\left(\sum_{i\in I}a_i\right),
\]
so in particular $Q_1=a_1\cdots a_k$ and $Q_2=\prod_{1\le i<j\le k}(a_i+a_j)$.
For $1\le n\le k$ define
\[
 P_{k,n}(a_1,\ldots,a_k):=\prod_{j=1}^n Q_j(a_1,\ldots,a_k)
 = Q_1 Q_2\cdots Q_n.
\]
It is clear that
$$
\deg Q_j=\binom{k}{j}, \quad \deg P_{k,n}
=\deg Q_1+\cdots+\deg Q_n=\alpha_n(k). \eqno(4.1)
$$ 

\begin{proof}[Proof of Theorem~\ref{Pknm}]
If $n=1$, the hypothesis is $Q_1^m\ne0$ in $\mathcal R_{d,k}$,
which is equivalent to $d\ge m+k-1$. The conclusion is
Proposition~\ref{prop:n1}. Assume henceforth that $n\ge2$.

\noindent{\bf Step 1: The polynomial condition.}
Write
\[
R=Q_1^{m-1}(Q_2\cdots Q_n)^m,
\qquad P_{k,n}^m=Q_1R.
\]
Dividing by $Q_1=a_1\cdots a_k$ subtracts one from every exponent.
Thus the hypothesis is equivalent to
\begin{equation}\label{eq:R-nonzero}\tag{4.2}
R\ne0\quad\text{in }\mathcal P_{d,k}.
\end{equation}

\noindent{\bf Step 2: The equivariant map on $V_{d,k}$ (smooth case).}
Assume first that each $\mu_\ell$ has a strictly positive smooth density.
For every unit vector $u\in\mathbb{S}^{d-1}$ there is a unique median offset
$h_\ell(u)$ satisfying $\mu_\ell\{x:\langle u,x\rangle\le h_\ell(u)\}=\mu_\ell(\mathbb{R}^d)/2$;
this function is continuous and odd: $h_\ell(-u)=-h_\ell(u)$.

On the Stiefel manifold $V_{d,k}$, the unit normals $u_1,\ldots,u_k$ are
already mutually orthogonal.  Using the first measure's median map
$h_1$, fix the offset $t_i=h_1(u_i)$ for each $i$; this is $B_k$-equivariant.
The remaining conditions ($|I|=1$ for $\mu_2,\ldots,\mu_m$, and $2\le|I|\le n$
for all $\mu_\ell$) define a continuous $B_k$-equivariant map
\[
 F: V_{d,k}\longrightarrow
 (W_1^+)^{\oplus(m-1)}
 \oplus\bigoplus_{j=2}^n(W_j^+)^{\oplus m},
\]
with Euler polynomial $R = Q_1^{m-1}(Q_2\cdots Q_n)^m$.
Condition~\eqref{eq:R-nonzero} and Theorem~\ref{thm:main-zero} imply
that $F$ has a zero, which by Theorem~\ref{th34} gives the required
equipartition.

\medskip

\noindent{\bf Step 3: Gaussian approximation (general case).}
For general admissible $\mu_\ell$, replace each by
$\mu_\ell*\mathcal{N}(0,\varepsilon^2 I)$.  These convolutions have smooth
positive densities, so by Step~2 there exists a frame
$\mathbf{u}^\varepsilon\in V_{d,k}$ giving the required $n$-wise
equipartition of all $m$ convolved measures.
The offsets $t_i^\varepsilon=h_1^\varepsilon(u_i^\varepsilon)$ are
uniformly bounded (the half-mass ball of $\mu_1*\mathcal{N}(0,\varepsilon^2I)$
has bounded radius, uniformly in $\varepsilon\le1$).
By compactness, a subsequence converges to some
$(\mathbf{u}^0,t^0)\in V_{d,k}\times\mathbb{R}^k$.
Represent the convolution on the fixed product measure
$\mu_\ell\otimes\mathcal N(0,I)$ by $(x,z)\mapsto x+\varepsilon z$.
The cell indicators converge almost everywhere, since the limiting
hyperplanes have $\mu_\ell$-measure zero. Dominated convergence therefore
passes all equipartition equalities to the limit.
\end{proof}

\subsection{Upper bound: the nonvanishing lemmas}

The algebraic upper bounds in cases~(i) and~(iii) of Theorem~\ref{th15},
as well as the special case $n=2$, $m=2t-1$, reduce via Theorem~\ref{Pknm}
to proving $(P_{k,n})^m \ne 0$ in $\mathcal R_{d^*,k}$
with $d^* = m + (\beta_n(k)-1)\,2^{\lfloor\log_2 m\rfloor}$.
For general $m$ and $n=2$, the geometric upper bound is taken
directly from~\cite{Mejia25}.
The following two auxiliary lemmas are used throughout.

\begin{lemma} \label{L43}
Let $2\le n \le k$.  The monomial
$$
M_0 := a_1^{\beta_n(k)}\,a_2^{\beta_{n}(k-1)}\cdots a_{k-1}^{\beta_{n}(2)}\,a_k^{\beta_n(1)}
$$
appears with nonzero coefficient in $P_{k,n}(a_1,\ldots,a_k)$
in the polynomial ring $\mathbb{F}_2[a_1,\ldots,a_k]$.
Consequently, $M_0$ is the leading monomial of $P_{k,n}$ in the
lexicographic order $a_1>\cdots>a_k$.
\end{lemma}

\begin{proof}
Write $P_{k,n}=Q_1 Q_2\cdots Q_n$.
In the lexicographic order $a_1>\cdots>a_k$,
the leading monomial of $Q_j=\prod_{|I|=j}(\sum_{i\in I}a_i)$
is obtained by taking, in each factor $\sum_{i\in I}a_i$,
the term $a_{\min I}$.  The factors of $Q_j$ indexed by all $j$-element
subsets $I\subset[k]$ then contribute $a_i^{\binom{k-i}{j-1}}$ to the
leading monomial, since $a_i = a_{\min I}$ precisely for those $I$
that contain $i$ but no element smaller than $i$, and there are
$\binom{k-i}{j-1}$ such subsets.  Hence the leading monomial of $Q_j$ is
$$
\prod_{i=1}^k a_i^{\binom{k-i}{j-1}}.
$$
Multiplying over $j=1,\ldots,n$ gives the leading monomial of $P_{k,n}$:
$$
\prod_{i=1}^k a_i^{\,\sum_{j=1}^n\binom{k-i}{j-1}}
=\prod_{i=1}^k a_i^{\,\beta_n(k-i+1)}
=a_1^{\beta_n(k)}\,a_2^{\beta_n(k-1)}\cdots a_k^{\beta_n(1)} = M_0.
$$
Since the leading monomial of a product equals the product of the leading
monomials (over any field), no cancellation occurs, so the coefficient
of $M_0$ in $P_{k,n}$ is $1\ne0$ in $\mathbb{F}_2$.
\end{proof}

\begin{lemma} \label{L44}
Let $2\le n \le k$, $0\le i\le k-1$, $q\ge0$, and $0\le r<2^q$.  Define
$$
d_i:=2^q\beta_{n}(k-i)+r\beta_{n}(i+1).
$$
Then $d_0-d_i \ge i$.
\end{lemma}
\begin{proof}
Since $\beta_n(i+1)\ge1$, the difference $d_0-d_i$ is minimised when $r$
is as large as possible, i.e.\ $r=2^q-1$.  Setting $s=2^q-1$ gives
$$
d_0-d_i \ge 2^q A(i)+B(i),
$$
where
$$
A(i):=\beta_n(k)+1-\beta_n(k-i)-\beta_n(i+1),
\quad
B(i):=\beta_n(i+1)-1=\sum_{j=1}^{n-1}\binom{i}{j}.
$$

\noindent\textit{Proof that $A(i)\ge0$.}
Write $a=k-i-1$ and $b=i$, so $a+b=k-1$ and $a,b\ge0$.
For every integer $h\ge1$ the Vandermonde identity gives
$\binom{a+b}{h}\ge\binom{a}{h}+\binom{b}{h}$
(the left side counts $h$-subsets of an $(a+b)$-set; the right side
counts only those contained entirely in one of the two parts).
Summing over $h=1,\ldots,n-1$ and using $\binom{a+b}{0}=\binom{a}{0}=\binom{b}{0}=1$:
$$
\beta_n(k)=\sum_{h=0}^{n-1}\binom{k-1}{h}
\ge\sum_{h=0}^{n-1}\binom{k-i-1}{h}+\sum_{h=0}^{n-1}\binom{i}{h}-1
=\beta_n(k-i)+\beta_n(i+1)-1,
$$
hence $A(i)\ge0$.  Equality holds at $i=0$ or $i=k-1$.

\noindent\textit{Proof that $B(i)\ge i$.}
Since $n\ge2$,
$B(i)=\sum_{j=1}^{n-1}\binom{i}{j}\ge\binom{i}{1}=i$.

Therefore $d_0-d_i\ge i$.
\end{proof}

\subsubsection*{Case 1: $n=k$ (the MVZ upper bound also holds orthogonally)}

\begin{lemma}\label{L:nk}
Let $k,m\ge1$ and
$U(m,k)=m+(2^{k-1}-1)2^{\lfloor\log_2 m\rfloor}$. Then
\[
(P_{k,k})^m\ne0\quad\text{in }\mathcal R_{U(m,k),k}.
\]
Consequently, $\Delta^*(m,k,k)\le U(m,k)$.
\end{lemma}

\begin{proof}
Write $D_k=P_{k,k}$. The Dickson determinant identity gives
\[
D_k=\prod_{\varnothing\ne I\subseteq[k]}\Bigl(\sum_{i\in I}a_i\Bigr)
=\det(a_i^{2^{j-1}})_{i,j=1}^k;
\]
see~\cite[Section~4]{MSZ}. We prove the assertion by induction on $k$,
the case $k=1$ being immediate. Put $m=t+r$, where $t=2^q$ and
$0\le r<t$, so that $U(m,k)=t2^{k-1}+r$.

Regarding $D_k$ as a polynomial in $a_1$, its highest and lowest
powers are $a_1^{2^{k-1}}$ and $a_1$, with respective coefficients
$D_{k-1}$ and $D_{k-1}^2$ in the remaining variables. We claim that
\begin{equation}\label{eq:dickson-coefficient}\tag{4.3}
[a_1^{t2^{k-1}+r}]D_k^{t+r}=D_{k-1}^{t+2r}.
\end{equation}
Indeed, factor $D_k^m$ according to the nonzero binary digits of $m$.
By Frobenius, each factor contributes a power of two to the exponent
of $a_1$. The desired exponent $t2^{k-1}+r$ has exactly as many
nonzero binary digits as $m$, so no carries can occur in their sum.
Its highest digit must come from $D_k^t$, using the term
$a_1^{t2^{k-1}}$. The minimum possible contribution from all remaining
factors is already $r$, so each must use its lowest power of $a_1$.
This proves~\eqref{eq:dickson-coefficient}.

Set $m'=t+2r$. Directly from the definition of $U$,
\[
U(m,k)-U(m',k-1)=
\begin{cases}
t2^{k-2}-r,&2r<t,\\
t-r,&2r\ge t.
\end{cases}
\]
Both differences are at least one. By induction,
$D_{k-1}^{m'}$ has a monomial with nonzero coefficient whose exponent
of $a_i$ is at most $U(m',k-1)-i+2$, for $2\le i\le k$.
Multiplying this monomial by $a_1^{U(m,k)}$ gives, by
\eqref{eq:dickson-coefficient}, a monomial of $D_k^m$ with exponent
of $a_i$ at most $U(m,k)-i+1$. It survives in
$\mathcal R_{U(m,k),k}$. Theorem~\ref{Pknm} gives the claimed bound.
\end{proof}

\subsubsection*{Case 2: $n=2$ (any $k\ge2$, any $m\ge1$)}

\begin{lemma}\label{L:n2}
Let $k\ge2$ and $m\ge1$.  With $t=2^{\lfloor\log_2 m\rfloor}$,
$r=m-t$, and $d^*=tk+r$:
$$
\Delta^*(m,k,2)\le d^* = m+(k-1)\,2^{\lfloor\log_2 m\rfloor}.
$$
For $m=2t-1$ (one less than a power of two) the bound follows from
the nonvanishing $(P_{k,2})^{2t-1}\ne0$ in
$\mathcal R_{d^*,k}$.
For the remaining values of $m$, the bound follows from~\cite[Theorem~3.2]{Mejia25}.
\end{lemma}

\begin{proof}
We first prove the polynomial assertion for $m=2t-1$.

Over $\mathbb{Z}$, let $V=\prod_{i<j}(a_i-a_j)$ be the Vandermonde
determinant and let $C$ be the coefficient of
$\prod_i a_i^{t(k-1)-(i-1)}$ in $V^{2t-1}$.
Dyson's constant-term identity~\cite{GoodDyson} says that the
coefficient of $\prod_i a_i^{t(k-1)}$ in $V^{2t}$ has absolute value
$(tk)!/(t!)^k$. Expand one factor $V$ as a determinant. Its $k!$ terms
give coefficients of $V^{2t-1}$ indexed by permutations of
$(t(k-1),t(k-1)-1,\ldots,t(k-1)-k+1)$. Since $V^{2t-1}$ is
alternating, the coefficient sign and the determinant sign cancel,
so these $k!$ contributions are equal. Therefore
\[
|C| = \frac{(tk)!}{k!\,(t!)^k}.
\]
To see that $|C|$ is odd, apply Legendre's formula
$\nu_2(N!)=N-s_2(N)$ where $s_2(N)$ is the sum of binary digits:
\[
\begin{aligned}
\nu_2(|C|)&=\nu_2((tk)!)-\nu_2(k!)-k\nu_2(t!)\\
&=(tk-s_2(tk))-(k-s_2(k))-k(t-1)=0,
\end{aligned}
\]
since $t=2^q$ implies $s_2(tk)=s_2(k)$.  Hence $C$ is odd.

Reducing mod~$2$, $V$ becomes $Q_2$, so the monomial
$\prod_i a_i^{t(k-1)-(i-1)}$ has coefficient~$1$ in $(Q_2)^{2t-1}$.
Multiplying by $Q_1^{2t-1}=(a_1\cdots a_k)^{2t-1}$, the staircase
monomial
$$
a_1^{t(k+1)-1}\cdot a_2^{t(k+1)-2}\cdots a_k^{t(k+1)-k}
$$
has coefficient~$1$ in $(P_{k,2})^{2t-1}$, and its exponents satisfy
the caps $(t(k+1)-1,t(k+1)-2,\ldots,t(k+1)-k)$.

For the remaining values $m=t+r$ with $0\le r\le t-2$, the upper bound
$\Delta^*(m,k,2)\le tk+r$ follows from \cite[Theorem~1.5(a)]{Mejia25},
with their parameters $q=\log_2 t$ and $t_{\mathrm{MSZ}}=t-r\ge2$.
Their Theorem~3.2 treats this range by a test map on a product of
spheres with additional orthogonality equations. Together with the
case $r=t-1$ proved above, this covers every $m$.
\end{proof}

\subsubsection*{Case 3: $m=2^q$ a power of two (any $2\le n\le k$)}

\begin{lemma}\label{L:pow2}
Let $d^*=2^q\beta_n(k)$.  Then
$(P_{k,n})^{2^q}\ne0$ in $\mathcal{R}_{d^*,k}$,
and consequently $\Delta^*(2^q,k,n)\le 2^q\beta_n(k)$.
\end{lemma}

\begin{proof}
By the Frobenius endomorphism of $\mathbb{F}_2[a_1,\ldots,a_k]$,
$(P_{k,n})^{2^q}$ has the same support as $\{M^{2^q}:M\in\mathrm{supp}(P_{k,n})\}$
without cancellation.  The leading monomial $M_0=a_1^{\beta_n(k)}\cdots a_k^{\beta_n(1)}$
of $P_{k,n}$ (Lemma~\ref{L43}) gives $M_0^{2^q}=a_1^{2^q\beta_n(k)}\cdots a_k^{2^q\beta_n(1)}$
with coefficient~$1$ in $(P_{k,n})^{2^q}$.
By Lemma~\ref{L44} with $r=0$, the exponent $2^q\beta_n(k-i+1)$
satisfies $2^q\beta_n(k-i+1)\le 2^q\beta_n(k)-i+1=d^*-i+1$ for each $i$,
so $M_0^{2^q}$ survives in $\mathcal{R}_{d^*,k}$.
By Theorem~\ref{Pknm}, $\Delta^*(2^q,k,n)\le d^*$.
\end{proof}

\noindent{\bf Proof of the forward doubling inequality $d^*(2m,k,n)\le 2\,d^*(m,k,n)$.}\label{thm:dstar-double}

\begin{proof}
Over $\mathbb{F}_2$ the Frobenius endomorphism gives
$(P_{k,n})^{2m}=((P_{k,n})^m)^2$ with no cross-terms: every monomial
of $(P_{k,n})^{2m}$ has the form $a^{2e}$ where $a^e$ is a monomial
of $(P_{k,n})^m$, with the same coefficient.

\textit{Forward direction ($d^*(2m)\le 2d^*(m)$).}
The cap at position $i$ ($1\le i\le k$) in $\mathcal{R}_{d,k}$ is
$d-i+1$, so $a^e$ survives iff $e_i\le d-i+1$ for all $i=1,\ldots,k$.
If $a^e$ survives in $\mathcal{R}_{d,k}$, i.e.\ $e_i\le d-i+1$,
then $2e_i\le 2d-2i+2\le 2d-i+1$ (since $i\ge1$),
so $a^{2e}$ survives in $\mathcal{R}_{2d,k}$.
Hence $(P_{k,n})^m\ne0$ in $\mathcal{R}_{d,k}$
implies $(P_{k,n})^{2m}\ne0$ in $\mathcal{R}_{2d,k}$,
giving $d^*(2m,k,n)\le 2\,d^*(m,k,n)$.

\end{proof}

\begin{remark}
The reverse inequality $d^*(2m)\ge2\,d^*(m)$ does not follow from this argument.
Indeed, from $2e_i\le D-i+1$ one deduces only $e_i\le(D-i+1)/2$, which does
not in general imply $e_i\le\lfloor D/2\rfloor-i+1$
(e.g.\ $D=6$, $i=3$, $e_i=2$: $2e_i=4\le D-i+1=4$ but $e_i=2>\lfloor D/2\rfloor-i+1=1$).
\end{remark}

\medskip

\begin{proof}[Proof of Theorem~\ref{th15}]
The lower bound~(1.3) is Lemma~\ref{L41}.

Cases~(i) ($n=k$) and~(iii) ($m=2^q$): the nonvanishing of $(P_{k,n})^m$ in
$\mathcal R_{d^*,k}$ is established in
Lemmas~\ref{L:nk} and~\ref{L:pow2} respectively, and
Theorem~\ref{Pknm} gives the upper bound.

Case (ii), $m=2t-1$: the nonvanishing $(P_{k,2})^{2t-1}\ne0$ is
established in Lemma~\ref{L:n2}, and Theorem~\ref{Pknm} gives the
upper bound.

Case (ii), $0\le r\le t-2$: the upper bound $\Delta^*(m,k,2)\le tk+r$
follows directly from~\cite[Theorem~1.5(a)]{Mejia25} as stated in
Lemma~\ref{L:n2}; here Theorem~\ref{Pknm} is not used.

The open case is addressed in Remark~\ref{rem:general-m}.
\end{proof}

The exact values quoted in the introduction follow from the identities
\[
\begin{aligned}
P_{4,3}^3&=a_1^{12}a_2^{11}a_3^{10}a_4^9
&&\text{in }\mathcal R_{12,4},\\
P_{4,3}^7&=a_1^{26}a_2^{25}a_3^{24}a_4^{23}
&&\text{in }\mathcal R_{26,4}.
\end{aligned}
\]
Both were obtained by expanding the defining product over $\mathbb F_2$.
Theorem~\ref{Pknm} gives the upper bounds $12$ and $26$; Lemma~\ref{L41}
gives the same lower bounds.

\begin{remark}\label{rem:general-m}
For $2<n<k$ and $m$ not a power of two, the nonvanishing
$(P_{k,n})^m\ne0$ in $\mathcal{R}_{d^*,k}$ with $d^*=t\beta_n(k)+r$
was checked for $k\le5$, $m\le11$, but a general proof is open.
The calculation expands the linear factors over $\mathbb F_2$, using
Frobenius for each binary digit of $m$, and discards a monomial as
soon as one of its exponents exceeds the corresponding bound in
$\mathcal R_{d^*,k}$. This is exact arithmetic; discarded monomials
cannot contribute at a later step. Put $M_1=M_0(a_k,\ldots,a_1)$. The monomial
$M_0^tM_1^r$ need not have nonzero coefficient: at $(k,n,m)=(5,3,3)$ its coefficient in $(P_{5,3})^3$
is zero (22 contributions cancel mod~2), though other monomials survive.
\end{remark}

\begin{proof}[Proof of Theorem~\ref{th17}]
\noindent{\em Upper bound.}
First suppose $n\ge2$ and $(Q_1P_{k,n})^m\ne0$ in $\mathcal R_{d,k}$.
As in the proof of Theorem~\ref{Pknm}, assume initially that the measures
have smooth positive densities and choose the offsets
$t_i=h_1(u_i)$ using the first measure's median function.
Let $c_\ell$ be the center of mass of $\mu_\ell$. The additional
equations are
\[
C_{\ell,i}(\mathbf u)=h_1(u_i)-\langle u_i,c_\ell\rangle=0,
\qquad 1\le\ell\le m,\quad 1\le i\le k.
\]
These functions are continuous and give $m$ copies of $W_1^+$.
Together with the remaining Walsh equations, they define a
$B_k$-equivariant map
\[
F:V_{d,k}\longrightarrow
(W_1^+)^{\oplus(2m-1)}\oplus\bigoplus_{j=2}^n(W_j^+)^{\oplus m}.
\]
The Euler polynomial of the target representation is
\[
E=Q_1^{2m-1}(Q_2\cdots Q_n)^m=\frac{(Q_1P_{k,n})^m}{Q_1}.
\]
The hypothesis gives $E\ne0$ in $\mathcal P_{d,k}$, so
Theorem~\ref{thm:main-zero} supplies a zero of $F$. All the resulting
hyperplanes pass through every $c_\ell$ and satisfy the required
partition equations. Gaussian convolution preserves the centers of
mass. Thus the approximation argument in Theorem~\ref{Pknm} proves
the assertion for all measures with finite first moments.

For $n=1$, the polynomial hypothesis is $d\ge2m+k-1$.
Put $L_0=\operatorname{span}(c_2-c_1,\ldots,c_m-c_1)$.
After choosing $i-1$ orthonormal normals, restrict the next normal to
\[
E_i=(L_0+\operatorname{span}(u_1,\ldots,u_{i-1}))^\perp.
\]
Then $\dim E_i\ge d-(m-1)-(i-1)\ge m+1$.
The $m$ mass differences across the hyperplane through $c_1$ with
normal $u\in S(E_i)$ form a continuous odd map into $\mathbb R^m$.
The Borsuk--Ulam theorem gives a zero. This selects $u_i$ orthogonal
to the previous normals, and the hyperplane passes through all
centers because $u_i\perp L_0$. Iterating proves the case $n=1$.

Finally, multiplication by $Q_1^m$ shifts every exponent by $m$.
Thus nonvanishing of $P_{k,n}^m$ in $\mathcal R_{D,k}$ implies
nonvanishing of $(Q_1P_{k,n})^m$ in $\mathcal R_{D+m,k}$, which proves
$\Delta^\oplus(m,k,n)\le d^*(m,k,n)+m$.

\medskip

\noindent{\em Lower bound: $\Delta^\oplus(m,k,n)\ge L+m$.}
We use the perturbations from Lemma~\ref{L41}, which are also defined
for $n=1$, and add the centroid conditions. As before we may assume $d\ge k$.  The configuration space is $X = V_{d,k}\times\mathbb{R}^k$
with $\dim X=dk-k(k-1)/2$, and we impose $N=m(a+k)$ conditions: the
$ma=m\alpha_n(k)$ Walsh conditions and $mk$ centroid-incidence conditions
(for measures with finite first moments)
$$
C_{\ell,i}(\rho,z)=\int_{\mathbb{R}^d}(\langle u_i,x\rangle-t_i)\,\rho_\ell(x)\,dx=0.
$$

For the Walsh perturbations use $h^z_I$ as in Lemma~\ref{L41}.
For the centroid directions define
$$
b^z_i(x)=\frac{\eta(x-x_+-\varepsilon u_i)-\eta(x-x_++\varepsilon u_i)}{2\varepsilon},
\quad \varepsilon=\tfrac14,
$$
with $x_+=\sum_j(t_j+1)u_j$ and a radially symmetric mass-one bump $\eta$ supported in the ball of radius $1/8$ about zero.  Then $\int b^z_i\,dx=0$,
$\int(\langle u_j,x\rangle-t_j)\,b^z_i\,dx=\delta_{ij}$, and
$\int x\,b^z_i\,dx=u_i$. Thus these perturbations vary the center
of mass in the prescribed normal directions.
The two translated bumps defining $b^z_i$ both lie in the same open
orthant determined by $z$ (their distance to each hyperplane $H_j(z)$
is at least $1-\varepsilon-1/8>0$), so $\phi_J(z,x)=+1$ on both supports.
Hence $\int\phi_J(z,x)b^z_i(x)\,dx=\int b^z_i\,dx=0$ for every
Walsh coordinate $J$.
The first moment of $h^z_{\{i\}}$ is $u_i$, whereas the first
moment of $h^z_I$ is zero for $|I|\ge2$; both statements follow by
summing the centers $x_s$ with their Walsh signs.
Setting $\hat h^z_{\{i\}}:=h^z_{\{i\}}-b^z_i$ and keeping $h^z_I$
for $|I|\ge2$, the joint family
$\{\hat h^z_I\}\cup\{b^z_1,\ldots,b^z_k\}$
has the identity Jacobian for the full $(a+k)$ coordinates at $z$.
Its rank remains $a+k$ on a neighborhood of $z$. A finite cover of
$K$ and sufficiently small coefficients again give positive smooth
probability densities; their first moments are finite.
The Sard argument of Lemma~\ref{L41} with $N=m(a+k)$ shows that
admissible measures exist for which no configuration satisfies all
conditions when $N>\dim X$, i.e.\ when
$$
d < \frac{m\alpha_n(k)}{k}+\frac{k-1}{2}+m.
$$
Hence
$$
\Delta^\oplus(m,k,n)\ge \left\lceil\frac{m\alpha_n(k)}{k}+\frac{k-1}{2}+m\right\rceil = L+m.
$$

For $n=2$ and $m=2^j-1$, Lemma~\ref{L:n2} proves
$d^*(m,k,2)\le2^{j-1}(k+1)-1=L$.
The lower bound and Theorem~\ref{Pknm} give the reverse inequality.
Hence the centroid bounds coincide and yield
\[
\Delta^\oplus(2^j-1,k,2)=L+(2^j-1)=2^{j-1}(k+3)-2.
\]
\end{proof}

\medskip

\noindent {\bf AI use statement.} LLMs were used for proof refinement, literature review, and manuscript revision.

 \medskip

 O. R. Musin,  University of Texas Rio Grande Valley, School of Mathematical and
 Statistical Sciences, One West University Boulevard, Brownsville, TX, 78520, USA.

 {\it E-mail address:} oleg.musin@utrgv.edu

\end{document}